\documentclass[11pt]{article}
\usepackage{amsmath,fullpage,amsthm,amssymb,xcolor,graphicx}
\usepackage{cite}
\usepackage{hyperref}
\usepackage{soul}
\newtheorem{prop}{Proposition}[section]
\newtheorem{theorem}{Theorem}[section]
\newtheorem{corollary}{Corollary}[section]
\newtheorem{lemma}{Lemma}[section]
\newtheorem{definition}{Definition}[section]
\newtheorem{remark}{Remark}[section]

\begin{document}

  \begin{center}
 {\Large \bf Unifying adjacency, Laplacian, and signless Laplacian theories}\\

 \vspace{6mm}

 \bf {Aniruddha Samanta$^a$, Deepshikha$^b$, Kinkar Chandra Das$^{c,*}$}

 \vspace{5mm}

{\it $^a$Theoretical Statistics and Mathematics Unit, Indian Statistical Institute,\\
Kolkata-700108, India\/} \\
{\tt E-mail: aniruddha.sam@gmail.com}\\[2mm]

{\it $^b$Department of Mathematics, Shyampur Siddheswari Mahavidyalaya, University of Calcutta,\\
 West Bengal 711312, India\/} \\
{\tt E-mail: dpmmehra@gmail.com}\\[2mm]

{\it $^c$Department of Mathematics, Sungkyunkwan University, \\
 Suwon 16419, Republic of Korea\/} \\
{\tt E-mail: kinkardas2003@gmail.com}

 \hspace*{30mm}


 \end{center}

 \baselineskip=0.23in

\begin{abstract} Let $G$ be a simple graph with associated diagonal matrix of vertex degrees $D(G)$, adjacency matrix $A(G)$, Laplacian matrix $L(G)$ and signless Laplacian matrix $Q(G)$. Recently, Nikiforov proposed the family of matrices $A_\alpha(G)$ defined for any real $\alpha\in [0,1]$ as $A_\alpha(G):=\alpha\,D(G)+(1-\alpha)\,A(G)$, and also mentioned that the matrices $A_\alpha(G)$ can underpin a unified theory of $A(G)$ and $Q(G)$. Inspired from the above definition, we introduce the $B_\alpha$-matrix of $G$, $B_\alpha(G):=\alpha A(G)+(1-\alpha)L(G)$ for $\alpha\in [0,1]$. Note that $ L(G)=B_0(G), D(G)=2B_{\frac{1}{2}}(G), Q(G)=3B_{\frac{2}{3}}(G), A(G)=B_1(G)$. In this article, we study several spectral properties of $ B_\alpha $-matrices to unify the theories of adjacency, Laplacian, and signless Laplacian matrices of graphs. In particular, we prove that each eigenvalue of $ B_\alpha(G) $ is continuous on $ \alpha $. Using this, we characterize positive semidefinite $ B_\alpha $-matrices in terms of $\alpha$. As a consequence, we provide an upper bound of the independence number of $ G $. Besides, we establish some bounds for the largest and the smallest eigenvalues of $B_\alpha(G)$. As a result, we obtain a bound for the chromatic number of $G$ and deduce several known results. 
In addition, we present a Sachs-type result for the characteristic polynomial of a $ B_\alpha $-matrix. 

\end{abstract}



\baselineskip=0.25in

\section{Introduction}
Throughout this article, we consider $G$ to be a simple undirected graph with vertex set $V(G)=\{ v_1, v_2, \ldots, v_n\}$ and edge set $E(G)$. The adjacency matrix of $G$ is a symmetric $n\times n$ matrix $A(G)$ whose $(i,j)$-th entry is $1$ if $v_i$ and $v_j$ are adjacent, $0$ otherwise. The degree matrix of $ G $ is a diagonal matrix $ D(G) $ whose $i^{th}$ diagonal entry is the degree of the vertex $ v_i $. The Laplacian and the signless Laplacian matrix of $ G $ are the matrices, $ L(G):=D(G)-A(G) $ and $ Q(G):=D(G)+A(G) $, respectively. If graph $ G $ is clear from the context, we simply write $ A,L, $ and $ D $ instead of $ A(G), L(G) $, and $ D(G) $, respectively. Adjacency, Laplacian, and signless Laplacian matrices are some of the matrices associated with a graph which are widely studied in the literature \cite{1,2,3,4,5,6,DA1,DA2,DA3}. It can be seen that many of the spectral properties of such matrices are quite different from each other. We will thus analyze the spectral properties of the convex combinations of $A(G)$ and $L(G)$ in order to understand how uniformly the spectral behavior transforms from one matrix to another.

\begin{definition}
	For $ \alpha \in [0,1] $ the $ B_\alpha $-matrix of $ G $ is the convex linear combination $B_{\alpha}(G):=\alpha A(G)+(1-\alpha)L(G)$ (or simply $ B_\alpha=\alpha A+(1-\alpha)L $, if $ G $ is clear from the context).
\end{definition}

\begin{remark}
	Note that $ L(G)=B_0(G),\, D(G)=2B_{\frac{1}{2}}(G),\, Q(G)=3B_{\frac{2}{3}}(G),\, A(G)=B_1(G)$.
\end{remark}

It is clear that the spectral properties of $ B_{1/2}(G) $ and $ B_{2/3}(G) $ are equivalent to the spectral properties of $ D(G) $ and $ Q(G) $, respectively. In fact, $ A(G), L(G), Q(G),$ and $ D(G) $ can be considered as the $ B_\alpha$-matrix of $ G $ up to proportionality. Therefore, on the one hand, the spectral properties of $ B_\alpha(G) $ may reveal the common connection among the spectral properties of all such well-known matrices. On the other hand, $ B_\alpha(G) $ may analyze the structural and combinatorial properties of the graph $ G $ in a better way.

Recently, Nikiforov \cite{Niki} introduced a family of matrices, known as $ A_\alpha $-matrix, which is a convex combination of $ A(G) $ and $ D(G) $. The theory of $ A_\alpha $-matrices merges the theories of the adjacency matrix and signless Laplacian matrix of graphs. Later on, much work has been done on these matrices. We have results on the spectral radius (\cite{AB1,A_spectral_radius_1, A_spectral_radius_2,Prof.K.C.Das-2}), the second largest eigenvalue \cite{A_2nd_largest_eigenvalue}, the $k$-th largest eigenvalue \cite{Prof.K.C.Das-2}, the least eigenvalue (\cite{Prof.K.C.Das-1,A_least_eigenvalue}), the multiplicity of the eigenvalues (\cite{alpha_4,AS1}), positive semidefiniteness \cite{PSD}, the characteristic polynomial \cite{A_ch_poly}, spectral determination of graphs \cite{A_graph_determined_by_spectrum}, etc. Motivated by Nikiforov's work, we consider $ B_\alpha $-matrices and study their spectral properties.  However, unlike $ A_\alpha $-matrices, $ B_\alpha $-matrices are not always non-negative, but they obey Perron-Frobenius type results. In this article, we develop the theory of $ B_\alpha $-matrices to unify the theory of adjacency matrix, Laplacian matrix, and signless Laplacian matrix.

The paper is organized as follows. In Section \ref{sec0}, we list some previously known results. Section \ref{sec1} discusses the positive semidefiniteness of $ B_\alpha $-matrices. As a consequence, we obtain an upper bound for the independence number. Then  we present some bounds of eigenvalues of $ B_\alpha(G) $ in terms of maximum degree and minimum degree, chromatic number, etc., in Section \ref{sec2}. Besides, we obtain a lower bound for chromatic number and derive several known results as consequences. Finally, we study the determinant and a Sachs-type result for the characteristic polynomial of $ B_\alpha(G) $ in Section \ref{sec3}. 

\section{Preliminaries}\label{sec0}

Let $ G $ be a simple graph with vertex set $ V(G)=\{ v_1, v_2, \dots, v_n\} $ and edge set $ E(G) $. If two vertices $ v_i $ and $ v_j $ are adjacent, we write $ v_i \sim v_j $, and $v_iv_j$ denotes the edge between them. The degree of a vertex $ v_i $  is the number of edges adjacent to $ v_i $ and is denoted by $ d_G(v_i) $ or simply $ d(v_i) $ or $ d_i $. The minimum degree and the maximum degree of graph $G$ are denoted by $\delta(G)$ (or simply $\delta$) and $\Delta(G)$ (or simply $\Delta$), respectively. The complement of a graph $ G $ is the graph $ \overline{G} $ with vertex set $ V(G) $ and two vertices in $ \overline{G}$  are adjacent if and only if they are non-adjacent in $ G $. The $ 0-1 $ incidence matrix of a graph $ G $ with $ n $ vertices $ \{v_1, v_2, \dots, v_n\} $ and $ m $ edges $ \{ e_1, e_2, \dots, e_m\} $ is an $ n\times m $ matrix $ M $ whose $ (i,j) $-th entry is $ 1 $ if the vertex $ v_i $ is incident on the edge $ e_j $ and $0$ otherwise.\\

The line graph of a graph $ G $ is the graph $ G_{\ell}$ with vertex set $ E(G)$ and two vertices $ e_i$ and $ e_j $ in $ G_{\ell} $ are adjacent if the edges $ e_i $ and $ e_j $ have a common vertex in the graph $ G $. The identity matrix of order $ n $ is denoted by $ I_n $ (or simply $ I $). An $ a\times b $ matrix whose entries are all ones is denoted by $ J_{a,b} $ (or simply $ J $ when the order is clearly understood). The transpose of a matrix $ M $ is denoted by $ M^t $.
\begin{lemma}{\rm \cite[Lemma 6.16]{Bapat}}\label{lm2.1}
	Let $ G $ be a graph with line graph $G_{\ell}$. If $ M $ is the $ 0-1 $ incidence matrix of $ G $, then $ M^tM=A(G_{\ell})+2I $. Moreover, if $ G $ is $k$-regular, then $ MM^t=A(G)+kI$.
\end{lemma}

Since the eigenvalues of any $ n\times n $ symmetric matrix $ S $ are real, we denote and ordered the eigenvalues of $ S $ as follows:
	\begin{equation}\label{eq1.1}
		\lambda_{max}(S)=\lambda_1(S)\geq\lambda_2(S)\geq \cdots \geq \lambda_n(S)=\lambda_{min}(S).
	\end{equation}
	For a graph $ G $, sometimes we denote and arrange the eigenvalues of $A(G),\, L(G)$ and $ Q(G) $ as $ \rho_1(G)\geq \rho_2(G)\geq \cdots \geq \rho_n(G)$, $ \mu_1(G)\geq \mu_2(G)\geq \cdots \geq\mu_n(G)$ and $ q_1(G)\geq q_2(G)\geq \cdots \geq q_n(G)$, respectively. Next, we present the well known Weyl Theorem.
\begin{theorem}{\rm \cite[Theorem 4.3.1]{horn-john2}}\label{th2.1}
	If $ S_1 $ and $ S_2$ are two Hermitian matrices of order $ n $ and their eigenvalues are ordered as in \eqref{eq1.1}. Then $$ \lambda_{n+1-i}(S_1+S_2)\leq  \lambda_{n+1-i-j}(S_1)+\lambda_{j+1}(S_2), ~~~i=1,2, \dots, n; ~~j=0,1,\dots, n-i .$$
\end{theorem}
\begin{corollary}\label{cor1}
		If $ S_1 $ and $ S_2$ are two Hermitian matrices of order $ n $ and their eigenvalues are ordered as in \eqref{eq1.1}. Then $$ \lambda_{n}(S_1)+\lambda_1(S_2)\leq \lambda_1(S_1+S_2)\leq   \lambda_{1}(S_1)+\lambda_{1}(S_2).$$
\end{corollary}
Let us recall the following upper bound of the largest eigenvalue of a Laplacian matrix.
\begin{theorem}{\rm \cite{DA0}}\label{th2.2}
Let $ G $ be a graph with the largest Laplacian eigenvalue $\mu_1(G)$. Then $$ \mu_1(G)\leq \max_{v_iv_j\in E(G)}\,(d_i+d_j).$$
\end{theorem}

\vspace*{2mm}
Let $ S= \left(\begin{array}{cc}
	S_{11} & S_{12}\\
	S_{21} & S_{22}
\end{array} \right)$ be a $ 2 \times 2 $ block matrix, where $ S_{11} $ and $ S_{22} $ are square matrices. If  $ S_{11} $ is nonsingular, then the Schur complement of $ S_{11} $ in $ S $ is defined to be the following matrix
\begin{equation}\label{eq1}
	S_{22}-S_{21}S_{11}^{-1}S_{12}.
\end{equation}

 Similarly, if $ S_{22} $ is nonsingular, then the Schur complement of $ S_{22} $ in $ S $ is $ S_{11}-S_{12}S_{22}^{-1}S_{21} $. Let us recall the Schur complement formula for the determinant.

\vspace*{2mm}
 \begin{theorem}{\rm \cite{Bapat}}\label{th2.3}
 Let $ S= \left(\begin{array}{cc}
 	S_{11} & S_{12}\\
 	S_{21} & S_{22}
 \end{array} \right)$ be a $ 2 \times 2 $ block matrix, where $ S_{11} $ and $ S_{22} $ are square matrices and $ S_{11} $ is nonsingular. Then $$ \det S=(\det S_{11})\,\det(S_{22}-S_{21}S_{11}^{-1}S_{12}).$$
 \end{theorem}

A matrix is irreducible if it is not similar via a permutation to a block upper triangular matrix (that has more than one block of positive size). Note that the adjacency matrix of a connected graph is always irreducible. Let $ S=(s_{ij})_{n\times m} $ be a matrix, then we denote $ |S|:=(|s_{ij}|)_{n\times m} $.

\begin{theorem}{\rm \cite[Theorem 6.2.24]{horn-john2}}\label{th2.4}
	A square matrix $ S $ of order $ n $ is irreducible if and only if $ (I+|S|)^{n-1}>0$, entry-wise.
\end{theorem}

\begin{theorem}{\rm \cite[Corollary 6.2.27]{horn-john2}}\label{th2.5}
Let $ S $ be an irreducibly diagonally dominant matrix of order $ n $. If $ S $ is Hermitian and every main diagonal entry is positive, then $ S $ is positive definite.	
\end{theorem}

\section{Positive semidefiniteness of $ B_\alpha(G) $}\label{sec1}
Let $ G $ be a graph with $ B_\alpha $-matrix $ B_\alpha(G):= \alpha A(G)+(1-\alpha)L(G)$. In this section we determine the values of $ \alpha $ for which the matrix $ B_\alpha(G) $ is positive semidefinite. As a consequence, we give an upper bound of the independence number. Results obtained in this section will be useful in the later sections.

 For simplicity, we use the notation $ B_\alpha $, $~ A ,~ L$, and $ D $ instead of $ B_\alpha(G),~ A(G), ~L(G)$, and $ D(G) $, respectively, when $ G $ is clear from the context. Some equivalent forms of $B_{\alpha}$ in terms of $A$, $ L $, and $D$ are as follows:
\begin{align*}
	B_{\alpha}&=\alpha A+(1-\alpha)L\\
	&=(2\alpha-1)A+(1-\alpha)D\\
	&=(1-2\alpha)L+\alpha D.
\end{align*}

Given two real matrices $ S=(s_{ij})_{m\times n} $ and $ M=(m_{ij})_{m\times n} $, we use the notation $ S\geq M $ if and only if $ s_{ij} \geq m_{ij} $ for all $ i,j $. For any connected graph $G$ with $n$ vertices and  $\alpha\,(\ne \frac{1}{2})\in[0,1]$, we have
\begin{align}\label{ieq1}
(I+|B_{\alpha}|)^{n-1}&=(I+|(2\alpha-1)A+(1-\alpha)D|)^{n-1}\nonumber \\
&=(I+|2\alpha-1|A+(1-\alpha)D)^{n-1} \nonumber \\
&\geq I+|2\alpha-1| A+\cdots+
|2\alpha-1|^{n-1}A^{n-1}.
\end{align}
Since $ G $ is connected, any two vertices $ v_i $ and $ v_j $ are joined by a path with length $ k\leq n-1 $. Therefore $ (i,j) $-th entry of $ A^k $, which counts the number of walks of length $ k $ connecting $ v_i $ and $ v_j$, is positive. Hence, from \eqref{ieq1}, $ (I+|B_{\alpha}|)^{n-1} \geq I+|2\alpha-1| A+\cdots+|2\alpha-1|^{n-1}A^{n-1}>0.$
Thus, by Theorem \ref{th2.4}, $B_{\alpha}$ is irreducible for $\alpha\,(\ne \frac{1}{2})\in[0,1]$.

	We begin the section with the following basic property of $ B_\alpha(G) $.
\begin{prop}
	For any $\alpha\in[0,1]$, eigenvalues of $B_{\alpha}(G)$ are real numbers.
\end{prop}
\begin{proof}
	Since $A(G)$ and $L(G)$ are symmetric matrices, so $B_{\alpha}(G)=\alpha A(G)+(1-\alpha)L(G)$ is also a symmetric matrix. Hence, all the eigenvalues of $B_{\alpha}$ are real numbers for any $\alpha\in[0,1]$.
\end{proof}

\vspace*{3mm}
For a graph $ G $, let $ \lambda_1(B_\alpha)\geq \cdots \geq \lambda_n(B_\alpha)$ be the eigenvalues of $ B_\alpha(G) $. In the following theorem, we prove that $\lambda_k(B_{\alpha})$ is uniformly continuous for $\alpha\in[0,1]$.

\begin{theorem}\label{Th3.1}
	Let $ B_\alpha $ be the $ B_\alpha $-matrix of a graph $ G $ with $ n $ vertices. Then, for $k\in\{1,2,\ldots,n\}$, the mapping $f_G:[0,1]\rightarrow\mathbb{R}$ defined as $f_G(\alpha)=\lambda_k(B_{\alpha})$ is a uniformly continuous function.	
\end{theorem}
\begin{proof}
	Let $ L$ and $ D $ be the Laplacian and the degree matrix of $ G $, respectively. Then, for any $\alpha,\beta\in[0,1]$, we have $B_{\alpha}-B_{\beta}=2(\beta-\alpha)L+(\alpha-\beta)D$. Then, for any $i\in\{1,2,\ldots,n\}$ and $j\in\{0,1,\ldots,n-i\}$, using Theorem \ref{th2.1}, we obtain
	\begin{align*}
		\lambda_{n+1-i}(B_{\alpha})&=\lambda_{n+1-i}(B_{\alpha}+B_{\beta}-B_{\beta})\\
		&\leq \lambda_{n+1-i-j}(B_{\alpha}+B_{\beta})+\lambda_{j+1}(-B_{\beta})\\
		&=\lambda_{n+1-i-j}(B_{\alpha}+B_{\beta})-\lambda_{n-j}(B_{\beta}).
	\end{align*}
	Thus, for any $i\in\{1,2,\ldots,n\}$ and $j\in\{0,1,\ldots,n-i\}$,
	\begin{align}
		\lambda_{n+1-i}(B_{\alpha})+\lambda_{n-j}(B_{\beta})&\leq \lambda_{n+1-i-j}(B_{\alpha}+B_{\beta}) \quad \text{ for all } \alpha,\beta\in[0,1] \label{eq3.1.1} .
	\end{align}
Assume that $\alpha\leq\beta$.	By using \eqref{eq3.1.1} and Corollary \ref{cor1}, we compute
	\begin{align*}
		\lambda_k(B_{\alpha})-\lambda_k(B_{\beta})
		&=\lambda_{n+1-k}(-B_{\beta})+\lambda_{n-(n-k)}(B_{\alpha})\\
		&\leq \lambda_{n+1-k-(n-k)}(-B_{\beta}+B_{\alpha})\\
		&= \lambda_{1}(B_{\alpha}-B_{\beta})\\
		&\leq \lambda_1(2(\beta-\alpha)L)+\lambda_1((\alpha-\beta)D)\\
		&=2(\beta-\alpha)\lambda_1(L)+(\alpha-\beta)\lambda_n(D)\\
		&\leq |\alpha-\beta|(2\lambda_1(L)+\lambda_1(D)).
	\end{align*}
	Also, by using \eqref{eq3.1.1} and Corollary \ref{cor1}, we obtain
	\begin{align*}
		\lambda_k(B_{\beta})-\lambda_k(B_{\alpha})&\leq\lambda_1(B_{\beta}-B_{\alpha})\\
		&\leq \lambda_1(2(\alpha-\beta)L)+\lambda_1((\beta-\alpha)D)\\
		&=2(\alpha-\beta)\lambda_n(L)+(\beta-\alpha)\lambda_1(D)\\
		&\leq|\alpha-\beta|(2\lambda_1(L)+\lambda_1(D)).
	\end{align*}
Thus, $|\lambda_k(B_{\beta})-\lambda_k(B_{\alpha})|\leq |\alpha-\beta|(2\lambda_1(L)+\lambda_1(D))$. Hence, the mapping $f_G$ is uniformly continuous.
\end{proof}

Note that for a graph $ G $ (with atleast an edge) on $ n$ vertices, $\lambda_n(B_\frac{1}{2})=\frac{1}{2}\lambda_n(D)\geq0$ and $ \lambda_n(B_1)=\lambda_n(A)\leq0 $. By Theorem \ref{Th3.1}, $ \lambda_n(B_\alpha) $ is continuous, so there exists a $ \beta\in (0,1) $ such that $ \lambda_n(B_\beta)=0 $. Therefore, for a graph $ G $ (with atleast an edge) on $ n$ vertices, we define $ \beta_o(G):=\max\{ \beta \in (0,1): \lambda_n(B_\beta)=0\}$. It is simply denoted by $ \beta_o $ if the graph $ G $ is understood from the context.

\begin{theorem}\label{Th3.2}
	If $G$ is a connected graph with $n\,(>1)$ vertices, then $B_{\alpha}$ is positive definite for $ \alpha \in(0,\frac{2}{3})$.
\end{theorem}
\begin{proof}
 First we assume that $0<\alpha\leq\frac{1}{2}$. Since $G$ is connected and $ \lambda_{n}(B_\alpha)=-\lambda_{1}(-B_\alpha)$, by Corollary \ref{cor1}, we obtain
\begin{align*}
	\lambda_n(B_{\alpha})&=\lambda_n((1-2\alpha)L+\alpha D)\\
	&\geq(1-2\alpha)\lambda_n(L)+\alpha\lambda_n(D)\\
	&=\alpha\lambda_n(D)>0.
\end{align*}
Thus, $B_{\alpha}$ is positive definite for $\alpha \in (0,\frac{1}{2}]$.

\vspace*{3mm}
\noindent Next we assume that $\frac{1}{2}<\alpha<\frac{2}{3}$. Then $0<2\alpha-1<1-\alpha$. Now $B_{\alpha}=(2\alpha-1)A+(1-\alpha)D$. Let $ (B_\alpha)_{ij} $ be the $ (i,j) $-th entry of $ B_\alpha$. Then, for $i\in\{1,2,\ldots,n\}$, we have
\begin{align*}
	|(B_{\alpha})_{ii}|&=|(1-\alpha)\,d_i|=(1-\alpha)\,d_i>(2\alpha-1)\,d_i=|2\alpha-1|\,d_i=\sum_{j=1,~ j\neq i}^n|(B_{\alpha})_{ij}|.
\end{align*}
Thus, $B_{\alpha}$ is strictly diagonally dominant with positive diagonal entries. Also, $B_{\alpha}$ is irreducible. Therefore, by Theorem \ref{th2.5}, $B_{\alpha}$ is positive definite for $\alpha \in (\frac{1}{2},\frac{2}{3})$.
\end{proof}
\begin{corollary}\label{Cor 3.2}
	If $G$ is a graph with no isolated vertices, then $B_{\alpha}$ is positive definite for\break $ \alpha \in(0,\frac{2}{3}) $.
\end{corollary}
Next corollary gives a lower bound of $\beta_o(G)$. For simplicity, we use $ \beta_o $ instead of $ \beta_o(G) $.
\begin{corollary}
	If $G$ is a graph with no isolated vertices, then $\beta_o\geq\frac{2}{3}$.
\end{corollary}
\begin{proof}
	By Theorem \ref{Th3.1}, $ f_G(\alpha):=\lambda_n(B_\alpha) $ is continuous. By Corollary \ref{Cor 3.2}, $ \lambda_n(B_\alpha) >0 $ for\break $ \alpha\in(0, \frac{2}{3}) $. Also, $ \lambda_n(B_1)<0 $. Therefore, $ \beta_o\geq \frac{2}{3} $.
\end{proof}

\begin{theorem}
 Let $ G $ be a graph with no isolated vertices. Then $ B_\alpha $ is positive semidefinite if and only if $ \alpha \in [0, \beta_o]$.
\end{theorem}
\begin{proof}
Let $ G $ be a graph with $ n $ vertices $ \{ v_1, v_2, \dots, v_n\} $. Set $ \beta_o:=\beta_o(G) $. Suppose $ \alpha\in [0,\beta_o] $. For $\alpha\in(0,\frac{2}{3})$, by Corollary \ref{Cor 3.2}, $B_{\alpha}$ is positive definite. Assume that $\frac{2}{3}\leq\alpha\leq\beta_o$. Then $-\alpha\geq-\beta_o$, that is, $\alpha\,(2\beta_o-1)\geq\beta_o\,(2\alpha-1)$ which implies $\frac{\alpha}{2\alpha-1}\geq\frac{\beta_o}{2\beta_o-1}$. Note that for any ${\bf x}=(x_1,x_2,\ldots,x_n)^t\in\mathbb{R}^n$, we obtain
\begin{align}\label{eq2}
	{\bf x}^t\,B_{\alpha}\,{\bf x}&\,=\,{\bf x}^t\,\alpha D\, {\bf x}+{\bf x}^t\,(1-2\alpha)L\,{\bf x} \nonumber\\
	&\,=\,\alpha\,\sum_{i=1}^n\,d_i\,x_i^2+(1-2\alpha)\,\sum_{v_jv_i\in E(G)}\,(x_i-x_j)^2.
\end{align}
Then, for any ${\bf x}=(x_1, x_2, \dots, x_n)^t\in\mathbb{R}^n$ and using \eqref{eq2}, we obtain
  $$0\leq {\bf x}^tB_{\beta_o}{\bf x}=\beta_o\,\sum\limits_{i=1}^n\,d_i\,x_i^2+\break (1-2\beta_o)\sum\limits_{v_j v_i\in E(G)}(x_i-x_j)^2$$
  which implies
  $$\sum\limits_{v_jv_i\in E(G)}(x_i-x_j)^2\leq\frac{\beta_o}{2\beta_o-1}\,\sum\limits_{i=1}^n\,d_i\,x_i^2\leq\frac{\alpha}{2\alpha-1}\,\sum\limits_{i=1}^n\,d_i\,x_i^2.$$
  Thus,
  $${\bf x}^tB_{\alpha}{\bf x}=\alpha\sum\limits_{i=1}^n\,d_i\,x_i^2+(1-2\alpha)\sum\limits_{v_jv_i\in E(G)}(x_i-x_j)^2\geq0 \,\mbox{ for any }~{\bf x}\in\mathbb{R}^n.$$
  Hence,
  $$\lambda_n(B_{\alpha})=\min\{{\bf x}^tB_{\alpha}{\bf x}:{\bf x}\in\mathbb{R}^n,\|{\bf x}\|=1\}\geq 0.$$
  Therefore, $B_{\alpha}$ is positive semidefinite for all $\alpha \in [0, \beta_o]$.

\vspace*{3mm}

To prove the converse, it is enough to prove that if $ \alpha \in (\beta_o,1] $, then $ B_\alpha $ is not positive semidefinite. Suppose there exists an $ \alpha \in(\beta_o, 1] $ such that $ \lambda_n(B_\alpha)\geq 0 $. By Theorem \ref{Th3.1}, the function $ f_G(\alpha):=\lambda_n(B_\alpha) $ is continuous and $ \lambda_n(B_1)<0 $, so there exist a $ \beta \in [\alpha, 1) $ such that $ \lambda_n(B_\beta)=0 $. This contradicts that $ \beta_o $ is the largest number in $ (0,1) $ for which $ \lambda_n(B_{\beta_o})=0 $. Hence, $ \lambda_n(B_\alpha)<0 $ for all $ \alpha \in (\beta_o,1]$.
\end{proof}

A symmetric matrix $ M $ is called indefinite if there exist two nonzero vectors $ x $ and $ y $ such that $ y^TMy >0> x^TMx$, where $ x^T $ denotes the transpose of $ x $.

\begin{corollary}
	For a graph $G$ with no isolated vertices, $B_{\alpha}$ is indefinite if and only if $\alpha\in(\beta_o,1]$.
\end{corollary}

\begin{proof}
	First we assume that $B_{\alpha}$ is indefinite. Set $ \beta_o:=\beta_o(G) $. Then $\lambda_1(B_{\alpha})>0$ and $\lambda_n(B_{\alpha})<0$. Thus, $\alpha\in(\beta_o,1]$.
	
	Conversely, suppose $\alpha\in(\beta_o,1]$. Then $\lambda_n(B_{\alpha})<0$. Also, $\alpha>\beta_o>\frac{1}{2}$. Hence, by Corollary \ref{cor1}, we have
	\begin{align*}
		\lambda_1(B_{\alpha})&=\lambda_1((2\alpha-1)A+(1-\alpha)D)\\
		&\geq\lambda_1((2\alpha-1)A)+\lambda_n((1-\alpha)D)\\
		&=(2\alpha-1)\lambda_1(A)+(1-\alpha)\lambda_n(D)\\
		&>0.
	\end{align*}
	Hence, 	$B_{\alpha}$ is indefinite.
\end{proof}
In the next theorem, we find $\beta_o$ for regular graphs.
\begin{theorem}
	If $G$ is an $r$-regular graph with smallest adjacency eigenvalue $ \rho_n(G) $. Then $$\beta_o(G)=\frac{r-\rho_n(G)}{r-2\rho_n(G)}.$$
\end{theorem}

\begin{proof}
Set $ \beta_o:=\beta_o(G) $. Then
\begin{align*}
	0=\lambda_n(B_{\beta_o})&=\lambda_n\Big((2\beta_o-1)A+(1-\beta_o)D\Big)\\
	&=\lambda_n\Big((2\beta_o-1)A+(1-\beta_o)rI\Big)\\
	&=(2\beta_o-1)\lambda_n(A)+(1-\beta_o)r.
\end{align*}
 This gives $\displaystyle{\beta_o=\frac{r-\lambda_n(A)}{r-2\lambda_n(A)}.}$
\end{proof}

Let us recall the following Hoffman bound of independence number $ \alpha(G) $ of an $ r $-regular graph $ G $ with $ n $ vertices in terms of $ \rho_{n}(G) $.
$$ \alpha(G)\leq -\frac{\rho_{n}(G)}{r-\rho_{n}(G)}\,n.$$
In light of the above result, we obtain a close relation between the independence number $ \alpha(G) $ of a regular graph $ G $ and $ \beta_o(G)$.

\begin{prop}
	Let $ G $ be an $ r $-regular graph of $ n $ vertices with independence number $ \alpha(G) $ and $ \beta_o:=\beta_o(G) $. Then $$ \alpha(G)\leq n\left(\frac{1-\beta_o}{\beta_o}\right).$$
\end{prop}


\section{On eigenvalues}\label{sec2}
The spectral radius of a matrix is the maximum value among the absolute values of all eigenvalues of that matrix.  For any $ B_\alpha $-matrix, we first show a partial Perron-Frobenius type result (that is, $ \lambda_1(B_\alpha) $ is the spectral radius of $ B_\alpha $). Then we compute the eigenvalues of $ B_\alpha(G) $ for complete graph and complete bipartite graph. Thereafter, we obtain some lower and upper bounds on the largest eigenvalue of $ B_\alpha (G)$ of graph $G$ in terms of $\Delta$ and $\delta$. As a consequence, we deduce some known results.  In addition, we establish an upper bound on the smallest eigenvalue of $ B_\alpha(G) $ in terms of the chromatic number. Finally, we derive a bound on the chromatic number in terms of $ \beta_o(G) $. 



It is to be observed that $B_{\alpha}$-matrices are not always non-negative. Therefore, Perron-Frobenius Theorem is not directly applicable. However, we can still conclude that the spectral radius of a $B_{\alpha}$-matrix is the same as its largest eigenvalue.
\begin{theorem}\label{th4.1}
		For any $ \alpha\in[0,1] $, the spectral radius of $ B_\alpha $ of a connected graph $ G $ is $ \lambda_1(B_\alpha) $.
\end{theorem}
\begin{proof}
	For $\alpha\in[0,\frac{1}{2}]$, $B_{\alpha}$ is positive semidefinite by Theorem \ref{Th3.2}. Hence, $\lambda_1(B_{\alpha})$ is the spectral radius of $ B_\alpha $. Let $\alpha\in(\frac{1}{2},1]$. Then $2\alpha-1>0$ and hence $B_{\alpha}=(2\alpha-1)A+(1-\alpha)D$ is a non-negative matrix. Also, 	 $B_{\alpha}$ is irreducible.  Therefore, by Perron-Frobenius Theorem, $\lambda_1(B_{\alpha})$ is the spectral radius of $ B_\alpha $.
\end{proof}

\subsection{The spectrum of $B_\alpha$-matrix for complete and complete bipartite graphs}

In this subsection, we determine the eigenvalues of $ B_\alpha $-matrices of the complete graph and the complete bipartite graph.
\begin{prop}
	If $G$ is a complete graph with $n$ vertices, then eigenvalues of $B_{\alpha}(G)$ are $(1-\alpha)n-\alpha$ with multiplicity $n-1$ and $(n-1)\alpha $ with multiplicity $1$.
\end{prop}

A complete bipartite graph with vertex partition size $a$ and $b$ is denoted by $K_{a,b}$. Since the eigenvalues of the adjacency matrix of a complete bipartite graph are known, so we compute the eigenvalues of $B_\alpha(K_{a,b})$ for $\alpha \in [0,1)$.
\begin{prop} \label{1ws1}
	For $ \alpha\in[0,1) $, the eigenvalues of $ B_\alpha(K_{a,b}) $ are $(1-\alpha)a$ with multiplicity $b-1$, $(1-\alpha)b$ with multiplicity $a-1$, and $\displaystyle{\frac{(1-\alpha)(a+b)\pm\sqrt{(1-\alpha)^2(a-b)^2+4(2\alpha-1)^2ab}}{2}}$.
\end{prop}

\begin{proof}
	Let $J_{a,b}$ and $J_{b,a}$ be the square matrices of order $a\times b$ and $b\times a$, respectively, of all ones. Then
	\begin{align*}
		B_{\alpha}(K_{a,b})=B_{\alpha}&=(2\alpha-1)A+(1-\alpha)D\\
		&=(2\alpha-1)\left(\begin{array}{cc}
			0 & J_{a,b}\\[3mm]
			J_{b,a} & 0
		\end{array} \right) +(1-\alpha)\left(\begin{array}{cc}
			b I_{a} & 0\\[3mm]
			0 & a I_{b}
		\end{array} \right)\\[3mm]
		&=
		\left(\begin{array}{cc}
			(1-\alpha)bI_a & (2\alpha-1)J_{a,b}\\[3mm]
			(2\alpha-1)J_{b,a} & (1-\alpha)aI_b
		\end{array} \right).
	\end{align*}
	For $i\in\{2,3,\ldots,a\}$, let the vector ${\bf x^{(i)}}=(x^i_1,x^i_2, \dots, x^i_{a+b})^t$ of order $a+b$ be defined as
	\begin{align*}
		x_j^i=\begin{cases}
			1\  &\text{ for } j=1,\\
			-1\ & \text{ for } j=i,\\
			0\  & \text{ otherwise}.
		\end{cases}
	\end{align*}
	Then $\{{\bf x^{(2)}},{\bf x^{(3)}},\ldots,{\bf x^{(a)}}\}$ is a linearly independent set of eigenvectors corresponding to the eigenvalue $(1-\alpha)b$. Thus, $B_{\alpha}$ has eigenvalue $(1-\alpha)b$ with multiplicity $a-1$.
	
	For $i\in\{a+2,a+3,\ldots,a+b\}$, let the vector ${\bf x^{(i)}}=(x^i_1,x^i_2, \dots, x^i_{a+b})^t$ of order $a+b$ be defined as
	\begin{align*}
		x_j^i=\begin{cases}
			1\ &\text{ for } j=a+1,\\
			-1\ &\text{ for } j=i,\\
			0\ &\text{ otherwise}.
		\end{cases}
	\end{align*}
	Then  $\{{\bf x^{(a+2)}},{\bf x^{(a+3)}},\ldots,{\bf x^{(a+b)}}\}$ is a linearly independent set of eigenvectors corresponding to the eigenvalue $(1-\alpha)a$. Thus, $B_{\alpha}$ has eigenvalue $(1-\alpha)a$ with multiplicity $b-1$.
	
	Let $\alpha\in[0,1)$. Then, by using Theorem \ref{th2.3}, we compute
	\begin{align*}
		\det(B_{\alpha})&=\begin{vmatrix}
			(1-\alpha)\,b\,I_a & (2\alpha-1)\,J_{a,b}\\[3mm]
			(2\alpha-1)\,J_{b,a} & (1-\alpha)\,a\,I_b
		\end{vmatrix}\\[3mm]
		&=\det\Big((1-\alpha)\,b\,I_a\Big)\,\det\left((1-\alpha)a\,I_b-(2\alpha-1)\,J_{b,a}\Big((1-\alpha)\,b\,I_a\Big)^{-1}(2\alpha-1)\,J_{a,b}\right)\\[3mm]
		&=(1-\alpha)^a\,b^a\det\left((1-\alpha)aI_b-(2\alpha-1)^2J_{b,a}\frac{1}{(1-\alpha)b}I_a J_{a,b}\right)\\[3mm]
		&=(1-\alpha)^a\,b^a\det\left((1-\alpha)aI_b-\frac{a(2\alpha-1)^2}{(1-\alpha)b}J_{b,b}\right)\\[3mm]
		&=(1-\alpha)^a\,b^a(1-\alpha)^{b-1}\,a^{b-1}\left((1-\alpha)\,a-\frac{a\,(2\alpha-1)^2}{(1-\alpha)}\right).
	\end{align*}
	
	\vspace*{3mm}
	
	Let $x$ and $y$ be the remaining eigenvalues of $B_{\alpha}(K_{a,b})$. Then
	\begin{align*}
		xy\,\Big((1-\alpha)\,a\Big)^{b-1}\Big((1-\alpha)\,b\Big)^{a-1}=(1-\alpha)^a\,b^a\,(1-\alpha)^{b-1}\,a^{b-1}\left((1-\alpha)\,a-\frac{a\,(2\alpha-1)^2}{(1-\alpha)}\right).
	\end{align*}
	Thus we obtain
	\begin{align}\label{eq1}
		xy&=b(1-\alpha)\left((1-\alpha)a-\frac{a(2\alpha-1)^2}{(1-\alpha)}\right)\nonumber\\[1mm]
		&=(1-\alpha)^2ab-(2\alpha-1)^2ab.
	\end{align}
	Since the sum of the eigenvalues is equal to the trace of the matrix, we obtain
	$$x+y+(b-1)(1-\alpha)a+(a-1)(1-\alpha)b=(1-\alpha)ab+(1-\alpha)ab,$$
	that is, $x=(1-\alpha)(a+b)-y$. Substitute the value of $x$ in \eqref{eq1}, we obtain
	$$y^2-(1-\alpha)(a+b)y+(1-\alpha)^2ab-(2\alpha-1)^2ab=0.$$
	This gives
	$$y=\frac{(1-\alpha)(a+b)\pm\sqrt{(1-\alpha)^2(a-b)^2+4(2\alpha-1)^2ab}}{2}.$$
	Hence the eigenvalues of $B_{\alpha}(K_{a,b})$ are  $(1-\alpha)a$ with multiplicity $b-1$, $(1-\alpha)b$ with multiplicity $a-1$ and $\displaystyle{\frac{(1-\alpha)(a+b)\pm\sqrt{(1-\alpha)^2(a-b)^2+4(2\alpha-1)^2ab}}{2}.}$
\end{proof}

\subsection{Bounds on the largest eigenvalue}
For a connected graph $ G $ with vertex set $ V(G)=\{ v_1, v_2, \dots, v_n\} $, the distance between two vertices $ v_i $ and $ v_j $ is denoted by $ d(v_i, v_j) $ and is defined to be the length of the shortest path between them. We now establish some lower and upper bounds on the largest eigenvalue of $ B_\alpha(G) $ of graph $G$.
\begin{theorem}
Let $G$ be a connected graph with the minimum degree $\delta$. Then, for any $ \alpha\in [0,1] $ $$ \lambda_1(B_\alpha) \geq \alpha \delta.$$
\end{theorem}

\begin{proof} First we assume that $\alpha\in[0,\frac{1}{2}]$, that is, $2\alpha-1\leq0$. Let $ {\bf x}=(x_1,x_2, \dots, x_n)^t \in \mathbb{R}^n$ be an eigenvector of $ B_\alpha$ corresponding to $ \lambda_{1}(B_\alpha) $ such that $x_k=\min\limits_{1\leq i\leq n}\,x_i<0$. Then
\begin{align*}
\lambda_{1}(B_\alpha)\,x_k&=(2\alpha-1)\sum\limits_{v_j:v_kv_j \in E(G)} x_j+(1-\alpha)d_kx_k\leq (2\alpha-1)\, d_kx_k+(1-\alpha)d_k\, x_k=\alpha d_kx_k.
\end{align*}
Therefore, $\lambda_{1}(B_\alpha)\geq \alpha d_k\geq \alpha \delta$.

Next we assume that $\alpha\in (\frac{1}{2},1]$, that is, $2\alpha-1>0$. Then $ B_\alpha $ is irreducible and non-negative. Therefore, by Perron-Frobenius Theorem, $B_{\alpha}$ has a Perron eigenvector ${\bf x}=(x_1,x_2, \dots, x_n)^t>0$ corresponding to the eigenvalue $\lambda_{1}(B_\alpha)$. Let $x_k=\min\limits_{1\leq i\leq n}\,x_i>0$. Then we have
\begin{align*}
\lambda_{1}(B_\alpha) x_k&=(2\alpha-1)\sum_{v_j:v_jv_k\in E(G)}x_j+(1-\alpha)d_kx_k\geq (2\alpha-1)d_kx_k+(1-\alpha)d_kx_k=\alpha d_kx_k.
\end{align*}
Thus, $\lambda_{1}(B_\alpha)\geq \alpha d_k \geq\alpha \delta$, for $\alpha\in (\frac{1}{2},1]$. Hence, $\lambda_1(B_\alpha)\geq \alpha\delta$ for all $\alpha\in[0,1]$.
\end{proof}

\vspace*{3mm}
Let $ N_G(v_1) $  denote the set of vertices of $ G $ which are adjacent to $ v_1 $. Let $N_G[v_1]:=N_G(v_1) \cup \{v_1\}$.

\vspace*{3mm}
\begin{theorem}\label{1d1} Let $G$ be a graph with at least one edge and maximum degree $\Delta$. Then, for any $\alpha\,(\neq \frac{1}{2})\in [0,1]$,
	$$\lambda_1(B_\alpha) \geq \frac{Y}{Z},$$
where $Y$ and $Z$ are given by
\begin{align}
Y=&\Big[\alpha^2\,(3\alpha-1)^2\,(\Delta+1)^2\,\Big(2\alpha\,m_2+(1-\alpha)\,m_3\Big)+(1-\alpha)\,(2\alpha-1)^2\,(2\Delta+5\alpha-3\alpha^2)^2\,\Delta\Big]\,P^2\nonumber\\[1mm]
&+4\,(2\alpha-1)^2\,(\Delta+1)\Big[(2\alpha-1)\,\Delta\,(2\Delta+5\alpha-3\alpha^2)+\alpha\,(3\alpha-1)\,(\Delta+1)\,m_3\Big]\,PQ\nonumber\\[3mm]
&+4\,(2\alpha-1)^2\,(\Delta+1)^2\,\Big[2\alpha\,m_1+(1-\alpha)\,(\Delta+m_3)\Big]\,Q^2,\label{FG1}
\end{align}
\begin{align}
Z=\Big[(2\alpha-1)^2\,(2\Delta+5\alpha-3\alpha^2)^2+\alpha^2\,(3\alpha-1)^2\,(\Delta+1)^2\,(n-\Delta-1)\Big]\,P^2+4\Delta\,(\Delta+1)^2\,(2\alpha-1)^2\,Q^2\label{FG2}
\end{align}
with
\begin{align}
P=(2\alpha-1)\,\Big((\alpha-1)\,(3\alpha-2)\,\Delta+2\Big)~\mbox{ and }~Q=16\alpha^2-6\alpha^3-10\alpha+2,\label{FG3}
\end{align}
and $m_1=|E(N_G(v_1))|$, is the number of edges in the set $N_G(v_1)$, $m_2=|E(V(G)\backslash N_G[v_1])|$, is the number of edges in the set $V(G)\backslash N_G[v_1]$, $m_3$ is the number of edges between $N_G(v_1)$ and $V(G)\backslash N_G[v_1]$, and the vertex $ v_1 $ has degree $ \Delta $.
\end{theorem}

\begin{proof} Let $v_1$ be the maximum degree vertex of degree $\Delta$ in $G$. Also let $S=\{v_2,v_3,\ldots,v_{\Delta+1}\}$ be the set of vertices adjacent to $v_1$ in $G$. Let ${\bf x}=(x_1,x_2,\ldots,x_n)^t$ be any non-zero vector. Using Rayleigh quotient, we obtain
\begin{align}
{\bf x}^t B_\alpha {\bf x}\leq \lambda_{1}(B_\alpha){\bf x}^t {\bf x},~\mbox{ that is, }~\lambda_{1}(B_\alpha)\geq \frac{2\,\alpha\,\sum\limits_{v_iv_j\in E(G)}\,x_ix_j+(1-\alpha)\,\sum\limits_{v_iv_j\in E(G)}\,(x_i-x_j)^2}{\sum\limits^n_{i=1}\,x^2_i}.\label{FG0}
\end{align}
We consider the following two cases:

\vspace*{2mm}

\noindent
${\bf Case\,1.}$ $(\alpha-1)\,(3\alpha-2)\,\Delta+2\neq 0$. In this case $P=(2\alpha-1)\,\Big((\alpha-1)\,(3\alpha-2)\,\Delta+2\Big)\neq 0$ as $\alpha\neq \frac{1}{2}$. Setting
\begin{equation}
x_i=\left\{
                                     \begin{array}{ll}
                                       \displaystyle{1-\frac{2-5\alpha+3\alpha^2}{2(\Delta+1)}} & \hbox{for $i=1$,} \\[5mm]
                                       \displaystyle{\frac{16\alpha^2-6\alpha^3-10\alpha+2}{(2\alpha-1)\,\Big((\alpha-1)\,(3\alpha-2)\,\Delta+2\Big)}} & \hbox{for $i=2,3,\ldots,\Delta+1$,}\\[7mm]
                                       \displaystyle{\frac{\alpha(3\alpha-1)}{2(2\alpha-1)}} & \hbox{Otherwise.}
                                     \end{array}
                                   \right.\label{1w2}
\end{equation}
Since $m_1=|E(N_G(v_1))|$, $m_2=|E(V(G)\backslash N_G[v_1])|$, and $m_3$ is the number of edges between $N_G(v_1)$ and $V(G)\backslash N_G[v_1]$, we obtain
\begin{align*}
\sum\limits_{v_iv_j\in E(G)}\,x_ix_j&=\frac{(2\Delta+5\alpha-3\alpha^2)\,(16\alpha^2-6\alpha^3-10\alpha+2)}{2(\Delta+1)\,(2\alpha-1)\,\Big((\alpha-1)\,(3\alpha-2)\,\Delta+2\Big)}\,\Delta\\[3mm]
&~~~~+\frac{(16\alpha^2-6\alpha^3-10\alpha+2)^2}{(2\alpha-1)^2\,\Big((\alpha-1)\,(3\alpha-2)\,\Delta+2\Big)^2}\,m_1+\frac{\alpha^2\,(3\alpha-1)^2}{4\,(2\alpha-1)^2}\,m_2\\[3mm]
&~~~~+\frac{\alpha(3\alpha-1)\,(16\alpha^2-6\alpha^3-10\alpha+2)}{2(2\alpha-1)^2\,\Big((\alpha-1)\,(3\alpha-2)\,\Delta+2\Big)}\,m_3,
\end{align*}

\begin{align*}
\sum\limits_{v_iv_j\in E(G)}\,(x_i-x_j)^2&=\left(\frac{(2\Delta+5\alpha-3\alpha^2)}{2(\Delta+1)}-\frac{16\alpha^2-6\alpha^3-10\alpha+2}{(2\alpha-1)\,\Big((\alpha-1)\,(3\alpha-2)\,\Delta+2\Big)}\right)^2\,\Delta\\[3mm]
&~~~~~+\left(\frac{16\alpha^2-6\alpha^3-10\alpha+2}{(2\alpha-1)\,\Big((\alpha-1)\,(3\alpha-2)\,\Delta+2\Big)}-\frac{\alpha(3\alpha-1)}{2(2\alpha-1)}\right)^2\,m_3,
\end{align*}
and
\begin{align*}
\sum\limits^n_{i=1}\,x^2_i&=\frac{(2\Delta+5\alpha-3\alpha^2)^2}{4(\Delta+1)^2}+\frac{(16\alpha^2-6\alpha^3-10\alpha+2)^2}{(2\alpha-1)^2\,\Big((\alpha-1)\,(3\alpha-2)\,\Delta+2\Big)^2}\,\Delta+\frac{\alpha^2(3\alpha-1)^2}{4(2\alpha-1)^2}\,(n-\Delta-1).
\end{align*}

\vspace*{3mm}

\noindent
Using the above results in (\ref{FG0}), we obtain
$$\lambda_1(B_\alpha) \geq \frac{Y}{Z}$$
as
\begin{align*}
2\,\alpha\,\sum\limits_{v_iv_j\in E(G)}\,x_ix_j+(1-\alpha)\,\sum\limits_{v_iv_j\in E(G)}\,(x_i-x_j)^2=\frac{Y}{4\,(2\alpha-1)^2\,(\Delta+1)^2\,P^2}
\end{align*}
and
\begin{align*}
\sum\limits^n_{i=1}\,x^2_i=\frac{Z}{4\,(2\alpha-1)^2\,(\Delta+1)^2\,P^2},
\end{align*}
where $Y$, $Z$ and $P$ are given by (\ref{FG1}), (\ref{FG2}) and (\ref{FG3}), respectively. Moreover, the equality holds if and only if ${\bf x}=(x_1,x_2,\ldots,x_n)^t$ is an eigenvector corresponding to the eigenvalue $\lambda_1(B_\alpha)$ of $B_\alpha$, where $x_i$ is given in (\ref{1w2}).

\vspace*{3mm}

\noindent
${\bf Case\,2.}$ $(\alpha-1)\,(3\alpha-2)\,\Delta+2=0$. In this case $P=(2\alpha-1)\,\Big((\alpha-1)\,(3\alpha-2)\,\Delta+2\Big)=0$. In this case
 $$Y=4\,(2\alpha-1)^2\,(\Delta+1)^2\,\Big[2\alpha\,m_1+(1-\alpha)\,(\Delta+m_3)\Big]\,Q^2~~\mbox{ and }~~Z=4\Delta\,(\Delta+1)^2\,(2\alpha-1)^2\,Q^2.$$
Setting
\begin{equation*}
x_i=\left\{
                                     \begin{array}{ll}
                                       0 & \hbox{for $i=1$,} \\[2mm]
                                       1 & \hbox{for $i=2,3,\ldots,\Delta+1$,}\\[2mm]
                                       0 & \hbox{Otherwise.}
                                     \end{array}
                                   \right.
\end{equation*}

\vspace*{2mm}

\noindent
Since $m_1=|E(N_G(v_1))|$, $m_2=|E(V(G)\backslash N_G[v_1])|$, and $m_3$ is the number of edges between $N_G(v_1)$ and $V(G)\backslash N_G[v_1]$, we obtain
\begin{align*}
\sum\limits_{v_iv_j\in E(G)}\,x_ix_j=m_1,~~\sum\limits_{v_iv_j\in E(G)}\,(x_i-x_j)^2=\Delta+m_3~\mbox{ and }~\sum\limits^n_{i=1}\,x^2_i=\Delta.
\end{align*}

\noindent
Using the above results in (\ref{FG0}), we obtain
$$\lambda_1(B_\alpha) \geq \frac{2\,\alpha\,m_1+(1-\alpha)\,(\Delta+m_3)}{\Delta}=\frac{Y}{Z}.$$
Moreover, the equality holds if and only if ${\bf x}=(0,\underbrace{1,\ldots,\,1}_{\Delta},\underbrace{0,\ldots,\,0}_{n-\Delta-1})^t$ is an eigenvector corresponding to the eigenvalue $\lambda_1(B_\alpha)$ of $B_\alpha$.
\end{proof}

\begin{corollary}\label{Cor4.1}
	 Let $G$ be a graph of order $n$ with $m$ edges and maximum degree $\Delta$. Then
     $$\rho_1(G)\geq \frac{2m}{n}$$
with equality if and only if $G$ is a regular graph.
\end{corollary}

\begin{proof} For adjacency matrix, $\alpha=1$, that is, $B_1=B_1(G)=A(G)$. For $\alpha=1$, from Theorem \ref{1d1}, we obtain
$$P=2=Q,~Y=32\,(\Delta+1)^2\,(\Delta+m_1+m_2+m_3)=32\,(\Delta+1)^2\,m~\mbox{ and }~Z=16\,(\Delta+1)^2\,n$$
and hence
$$\rho_1(G)=\lambda_1(B_1)\geq \frac{Y}{Z}=\frac{2m}{n}.$$
Moreover, the equality holds if and only if ${\bf x}=\left(1,1,\ldots,1\right)^t$ is an eigenvector corresponding to the eigenvalue $\lambda_1(B_1)\,(=\rho_1(G))$ of $B_1$, that is, if and only if $G$ is a regular graph.
\end{proof}

\begin{corollary} Let $G$ be a graph of order $n$ with $m$ edges and maximum degree $\Delta$. Then
     $$\mu_1(G)\geq \Delta+1.$$
If $G$ is connected, then the above equality holds if and only if $\Delta=n-1$.
\end{corollary}

\begin{proof} For Laplacian matrix, $\alpha=0$, that is, $B_0=B_0(G)=L(G)$. For $\alpha=0$, from Theorem \ref{1d1}, we obtain $$P=-2\,(\Delta+1),~Q=2,~Y=16\,\Delta(\Delta+1)^2\,\Big((\Delta+1)^2+\frac{m_3}{\Delta}\Big)~\mbox{ and }~Z=16\,\Delta\,(\Delta+1)^3.$$
and hence
 $$\mu_1(G)=\lambda_1(B_0)\geq \frac{Y}{Z}=\Delta+1+\frac{m_3}{\Delta\,(\Delta+1)}\geq \Delta+1$$
as $m_3\geq 0$.

Suppose that $G$ is connected. Then the equality holds if and only if \break ${\bf x}=\left(\displaystyle{\frac{\Delta}{\Delta+1}},\underbrace{-\frac{1}{\Delta+1},\ldots,-\frac{1}{\Delta+1}}_{\Delta},\underbrace{0,\ldots,0}_{n-\Delta-1}\right)^t$ is an eigenvector corresponding to the eigenvalue $\lambda_1(B_0)\,(=\mu_1(G))$ of $B_0$ and $m_3=0$. Since $G$ is connected, then $\Delta=n-1$. If $\Delta=n-1$, then one can easily see that $\mu_1(G)=n=\Delta+1$. This completes the proof of the result.
\end{proof}

\begin{corollary}\label{Cor4.3}
	 Let $G$ be a graph of order $n$ with $m$ edges and maximum degree $\Delta$. Then
     $$q_1(G)\geq \frac{4m}{n}$$
with equality if and only if $G$ is a regular graph.
\end{corollary}

\begin{proof} For signless Laplacian matrix, $\alpha=\frac{2}{3}$, that is, $B_{2/3}=B_{2/3}(G)=\frac{1}{3}\,Q(G)$. For $\alpha=\frac{2}{3}$, from Theorem \ref{1d1}, we obtain $$P=\frac{2}{3}=Q,~Y=\frac{64}{243}\,(\Delta+1)^2\,(\Delta+m_1+m_2+m_3)=\frac{64}{243}\,(\Delta+1)^2\,m,~\mbox{ and }~Z=\frac{16}{81}\,(\Delta+1)^2\,n$$
and hence
 $$\frac{1}{3}\,q_1(G)=\lambda_1(B_{2/3})\geq \frac{Y}{Z}=\frac{4m}{3n},~\mbox{ that is, }~q_1(G)\geq \frac{4m}{n}.$$
Moreover, the equality holds if and only if ${\bf x}=\left(1,1,\ldots,1\right)^t$ is an eigenvector corresponding to the eigenvalue $\lambda_1(B_{2/3})\,(=\frac{1}{3}\,q_1(G))$ of $B_{2/3}$, that is, if and only if $G$ is a regular graph.
\end{proof}
The lower bounds found in Corollaries \ref{Cor4.1}-\ref{Cor4.3} are classical. One can find all of them in \cite{ZS}.

\begin{remark}	{\em
		In the following Table, we give a comparison between the exact value of $\lambda_1(B_\alpha)$ (Proposition \ref{1ws1}) and the lower bound on $\lambda_1(B_\alpha)$ obtained in Theorem \ref{1d1} for the graph $G=K_{1,24}$.
		\begin{center}
			\vspace{6pt}
			\begin{tabular} { | c | c | c | }
				\hline
			  	         \hspace{5mm}$ \alpha $ \hspace{5mm}   &  \hspace{2mm}  Exact value of  $\lambda_1(B_\alpha) $ \hspace{2mm}  & \hspace{2mm}  $\frac{Y}{Z}$  \hspace{2mm}\\[2mm]
				\hline
				 $0$   & $ 25 $ & $ 25$ \\[0.2mm]
				\hline
				$0.1$ & $ 22.317 $ & $22.317$ \\[0.2mm]
				\hline
				$ 0.2 $ & $ 19.658 $ & $ 19.654$ \\[0.2mm]
				\hline
				$ 0.3 $ & $ 17.035 $ & $ 16.997$ \\[0.2mm]
				\hline
				$ 0.4 $ & $ 14.469 $ & $ 13.675$ \\[0.2mm]
				\hline
				$ 0.6 $ & $ 9.703$ & $ 1.978$ \\[0.2mm]
				\hline
				$ 0.7 $ & $ 7.718 $ & $0.936 $ \\[0.2mm]
				\hline
				$ 0.8 $ & $ 6.232 $ & $ 0.250$ \\[0.2mm]
				\hline
				$ 0.9 $ & $ 5.334 $ & $ 0.361$ \\[0.2mm]
				\hline
				$ 1 $ & $ 4.899 $ & $ 1.92$ \\[0.2mm]
				\hline
			\end{tabular}\\			
			\vspace{4pt}
			Table 1. Comparison of the largest eigenvalue $\lambda_1(B_\alpha)$ and $\frac{Y}{Z}$.			
		\end{center}		
	}	
\end{remark}

Next, we observe that the following upper bound is continuous on $ \alpha $.
\begin{theorem}
Let $G$ be a connected graph with maximum degree $\Delta$. Then, for any $\alpha\in[0,1]$,
\begin{align*}
\lambda_1\,(B_\alpha)\leq\begin{cases}
(2-3\,\alpha)\,\Delta\ \text{  if}~\alpha\in[0,\frac{1}{2}],\\[3mm]
\alpha\,\Delta\ ~~~~~~~~\text{   if}~\alpha\in(\frac{1}{2},1].
\end{cases}
\end{align*}
\end{theorem}

\begin{proof}
First we assume that $\alpha\in[0,\frac{1}{2}]$. Then, by using $\alpha\geq 0$, $(1-2\alpha)\geq 0$, Corollary \ref{cor1}, and Theorem \ref{th2.2}, we have
\begin{align*}
\lambda_1(B_\alpha)=\lambda_{1}\Big(\alpha D+(1-2\alpha)L\Big) &\leq\lambda_{1}(\alpha D)+\lambda_{1}\Big((1-2\alpha)L\Big)\leq \alpha \Delta +(1-2\alpha)(2\Delta).
\end{align*}
That is, $\lambda_1(B_\alpha)\leq (2-3\alpha)\Delta$ for all $\alpha\in[0,\frac{1}{2}]$.

\vspace*{3mm}

Next we assume that $\alpha\in (\frac{1}{2},1]$. That is, $2\alpha-1>0$. Let ${\bf x}=(x_1,x_2, \dots, x_n)^t$ be an eigenvector of $B_{\alpha}$ corresponding to the eigenvalue $\lambda_{1}(B_{\alpha})$ such that $x_k=\max\limits_{1\leq i\leq n}\,x_i>0$. Then we obtain
\begin{align*}
\lambda_{1}(B_{\alpha})\,x_k\,&\,=(2\alpha-1)\sum_{v_j:v_jv_k \in E(G)}x_j+(1-\alpha)\,d_k\,x_k\\[1mm]
&\,\leq\,(2\alpha-1)\,d_k\,x_k+(1-\alpha)\,d_k\,x_k\\[1mm]
&\,=\,\alpha d_k\,x_k\\[1mm]
&\,\leq\, \alpha \Delta x_k.
\end{align*}
Therefore, $\lambda_1(B_{\alpha})\leq \alpha\Delta$ for all $\alpha\in(\frac{1}{2},1]$.
\end{proof}

For $\alpha \in [0,1]$ and non-negative integers $a $ and $b$, define $ f_\alpha(a,b)=\frac{(1-\alpha)\,(a+b)+\sqrt{(1-\alpha)^2\,(a-b)^2+4(2\alpha-1)^2\,ab}}{2}.$

\begin{theorem} Let $G=(U,\,W,\,E)$ be a connected bipartite graph, where $|U|=a,$ $|W|=b$. For $\alpha \in [0,1]$,
	\begin{equation}
		\lambda_1(B_{\alpha})\leq f_\alpha(a,b) \label{1dh1}
	\end{equation}
	with equality if and only if $G\cong K_{a,b}$.
\end{theorem}

\begin{proof} Without loss of generality, assume that $a\geq b$. Now, from Proposition \ref{1ws1}, we have $f_\alpha(a,b)=\lambda_1(B_\alpha(K_{a,b}))$. Since, by definition, $B_0=B_0(G)=L(G)$ and $G\subseteq K_{a,b}$, by (edge) interlacing and Proposition \ref{1ws1}, $$\lambda_1(B_0)=\mu_1(G)\leq a+b=f_0(a,b).$$ Since $B_{\frac{1}{2}}=B_{\frac{1}{2}}(G)=\frac{1}{2}\,D(G)$, we have $$	\lambda_1(B_{\frac{1}{2}})=\frac{1}{2}\Delta(G)\leq \frac{1}{2}\Delta(K_{a,b})=\frac{1}{2}\,a=f_{\frac{1}{2}}(a,b).$$
	
	As $B_{1}=B_{1}(G)=A(G)$, by (vertex) interlacing and Proposition \ref{1ws1}, we have
	$$	\lambda_1(B_{1})=\rho_1(G)\leq \rho_1(K_{a,b})=f_1(a,b).$$
	
	So we have to prove the result in (\ref{1dh1}) for $0<\alpha<\frac{1}{2}$ and $\frac{1}{2}<\alpha<1$. Let ${\bf x}=(x_1,x_2,\ldots,x_n)^T$ be an eigenvector corresponding to the largest eigenvalue $\lambda_1(B_{\alpha})$ of $B_\alpha$. Then $B_\alpha{\bf x}=	\lambda_1(B_{\alpha}){\bf x}$. We consider two cases:
	
	\vspace*{3mm}
	
	\noindent
	${\bf Case\,1.}$ $\frac{1}{2}<\alpha<1$. Let $x_i=\max\limits_{1\leq k\leq n}\,x_k$. Without loss of generality, we can assume that $v_i\in U$. Let $x_j=\max\limits_{v_k\in W}\,x_k$. For $v_i\in U$, we obtain
	$$\lambda_1(B_{\alpha})\,x_i=(1-\alpha)\,d_ix_i+(2\alpha-1)\,\sum\limits_{v_k:v_iv_k\in E(G)}\,x_k\leq (1-\alpha)\,d_ix_i+(2\alpha-1)\,d_ix_j,$$
	that is,
	\begin{equation}
		\Big[\lambda_1(B_{\alpha})-(1-\alpha)\,d_i\Big]\,x_i\leq (2\alpha-1)\,d_ix_j.\label{1r1}
	\end{equation}
	Similarly, for $v_j\in W$, we obtain
	\begin{equation}
		\Big[\lambda_1(B_{\alpha})-(1-\alpha)\,d_j\Big]\,x_j\leq (2\alpha-1)\,d_jx_i.\label{1r2}
	\end{equation}
	
	From the above two results, we obtain
	\begin{align}
		&\Big[\lambda_1(B_{\alpha})-(1-\alpha)\,b\Big]\,\Big[\lambda_1(B_{\alpha})-(1-\alpha)\,a\Big]\nonumber\\
		\leq &\Big[\lambda_1(B_{\alpha})-(1-\alpha)\,d_i\Big]\,\Big[\lambda_1(B_{\alpha})-(1-\alpha)\,d_j\Big]\nonumber\\
		\leq &(2\alpha-1)^2\,d_i\,d_j\leq (2\alpha-1)^2\,ab.\label{1r3}
	\end{align}
	Thus we obtain
	$$\lambda_1(B_{\alpha})^2-(1-\alpha)\,(a+b)\lambda_1(B_{\alpha})+(1-\alpha)^2\,ab-(2\alpha-1)^2\,ab\leq 0,$$
	that is,
	$$\lambda_1(B_{\alpha})\leq \frac{(1-\alpha)\,(a+b)+\sqrt{(1-\alpha)^2\,(a-b)^2+4(2\alpha-1)^2\,ab}}{2}=f_\alpha(a,b).$$
	The first part of the proof is done.
	
	\vspace*{3mm}
	
	Suppose that equality holds. Then all inequalities in the above argument must be equalities. From equalities in (\ref{1r1}) and (\ref{1r2}), we obtain $x_k=x_i$ for all $v_k\in N_G(v_j) \subseteq U$ and $x_{\ell}=x_j$ for all $v_{\ell}\in N_G(v_i)\subseteq W$. From equality in (\ref{1r3}), we obtain
	$d_i=b$ and $d_j=a$. Since $G$ is a connected bipartite graph, one can easily prove that $x_k=x_i$ for all $v_k\in U$ and $x_{\ell}=x_j$ for all $v_{\ell}\in W$. For $v_i,\,v_k\in U$, we have
	$$\lambda_1(B_{\alpha})\,x_i=(1-\alpha)\,d_ix_i+(2\alpha-1)\,\sum\limits_{v_k:v_iv_k\in E(G)}\,x_k=b\,\Big[(1-\alpha)\,x_i+(2\alpha-1)\,x_j\Big]$$
	and
	$$\lambda_1(B_{\alpha})\,x_i=d_k\,\Big[(1-\alpha)\,x_i+(2\alpha-1)\,x_j\Big].$$
	Thus we have
	$$b\,\Big[(1-\alpha)\,x_i+(2\alpha-1)\,x_j\Big]=d_k\,\Big[(1-\alpha)\,x_i+(2\alpha-1)\,x_j\Big],$$
	that is,
	$$\Big(b-d_k\Big)\,\Big[(1-\alpha)\,x_i+(2\alpha-1)\,x_j\Big]=0.$$
	Since all the elements in $B_{\alpha}$ are non-negative, by Perron-Frobenius theorem in matrix theory, we obtain that all the eigencomponents corresponding to the spectral radius $\lambda_1(B_{\alpha})$ are non-negative. Since $G$ is connected, $x_i\geq x_j>0$. From the above with $\frac{1}{2}<\alpha<1$, we must have $d_k=b$ for any $v_k\in U$. Similarly, $d_{\ell}=a$ for any $v_{\ell}\in W$. Hence $G\cong K_{a,b}$.
	
	\vspace*{3mm}
	
	\noindent
	${\bf Case\,2.}$ $0<\alpha<\frac{1}{2}$. Let $x_i=\max_{1\leq k\leq n}\,x_k$. Without loss of generality, we can assume that $v_i\in U$. Let $x_j=\min_{v_k:v_iv_k\in E(G)}\,x_k$. For $v_i\in U$, we obtain
	$$\lambda_1(B_{\alpha})\,x_i=(1-\alpha)\,d_ix_i+(2\alpha-1)\,\sum\limits_{v_k:v_iv_k\in E(G)}\,x_k\leq (1-\alpha)\,d_ix_i+(2\alpha-1)\,d_ix_j,$$
	that is,
	$$\Big[\lambda_1(B_{\alpha})-(1-\alpha)\,d_i\Big]\,x_i\leq (2\alpha-1)\,d_ix_j.$$
	Similarly, for $v_j\in W$, we obtain
	$$\Big[\lambda_1(B_{\alpha})-(1-\alpha)\,d_j\Big]\,x_j\geq (2\alpha-1)\,d_jx_i.$$
	Since $\alpha<\frac{1}{2}$, from the above two results, we obtain
	\begin{align*}
		&\Big[\lambda_1(B_{\alpha})-(1-\alpha)\,d_i\Big]\,\Big[\lambda_1(B_{\alpha})-(1-\alpha)\,d_j\Big]\,x_i\\
		\leq & (2\alpha-1)\,d_i\,\Big[\lambda_1(B_{\alpha})-(1-\alpha)\,d_j\Big]\,x_j\\
		\leq & (2\alpha-1)^2\,d_i\,d_j\,x_i,
	\end{align*}
	that is,
	$$\Big[\lambda_1(B_{\alpha})-(1-\alpha)\,d_i\Big]\,\Big[\lambda_1(B_{\alpha})-(1-\alpha)\,d_j\Big]\leq (2\alpha-1)^2\,d_i\,d_j,$$
	that is,
	\begin{align*}
		\Big[\lambda_1(B_{\alpha})-(1-\alpha)\,b\Big]\,\Big[\lambda_1(B_{\alpha})-(1-\alpha)\,a\Big]\leq (2\alpha-1)^2\,ab.
	\end{align*}
	Thus we obtain
	$$\lambda_1(B_{\alpha})^2-(1-\alpha)\,(a+b)\lambda_1(B_{\alpha})+(1-\alpha)^2\,ab-(2\alpha-1)^2\,ab\leq 0,$$
	that is,
	$$\lambda_1(B_{\alpha})\leq \frac{(1-\alpha)\,(a+b)+\sqrt{(1-\alpha)^2\,(a-b)^2+4(2\alpha-1)^2\,ab}}{2}=f_{\alpha}(a,b).$$
	
	\vspace*{3mm}
	
	Suppose that equality holds. Similarly as ${\bf Case\,1}$, one can easily prove that $G\cong K_{a,b}$.
	
	\vspace*{3mm}
	
	Conversely, let $G\cong K_{a,b}$. By Proposition \ref{1ws1}, we obtain
	$$\lambda_1(B_{\alpha})=\frac{(1-\alpha)\,(a+b)+\sqrt{(1-\alpha)^2\,(a-b)^2+4(2\alpha-1)^2\,ab}}{2}=f_{\alpha}(a,b).$$
This completes the proof of the theorem.
\end{proof}

\subsection{Bounds on the smallest eigenvalue}
We establish an upper bound on the smallest eigenvalue of a $B_\alpha$-matrix in terms of the chromatic number. Then, we characterize the extremal graphs for some cases. Finally, some known results are derived as a consequence.

\begin{theorem}\label{Th4.4}
	Let $ G $ be a graph of $ n$ vertices, $ m ~(>0)$ edges and chromatic number $ \chi$. Then, for any $ \alpha \in [0,1] $,
	\begin{equation}\label{eq5}
		\lambda_n\,(B_\alpha)\,\leq\, \frac{2\,m}{n}\,\left(\frac{\chi\,(1-\alpha)-\alpha}{\chi-1}\right).
	\end{equation}
\end{theorem}
\begin{proof}
	Set $B_\alpha:=B_\alpha(G)$. Partition the vertex set $ V(G) $ as $ \chi $ number of subsets $ V_1, V_2, \dots, V_{\chi} $, where each subset contains the vertices having the same colour. For each $ j\in\{ 1,2,\dots, \chi\} $, define a vector ${\bf x}:=(x_1,x_2, \dots,x_n)^t$ as follows:
	$$
	x_i = \left\{ \begin{array}{cc}
		\chi-1 & \mbox{ for $v_i\in V_j$} \\[3mm]
		-1 & \mbox{ otherwise.}
	\end{array}\right.$$
	
	Then
$$||{\bf x}||\,^2=\sum\limits_{i=1}^{n}x_i^2=(\chi-1)^2\,|V_j|+(n-|V_j|)=\chi\,(\chi-2)\,|V_j|+n.$$
For $j\in \{ 1,2, \dots, \chi\}$, define $m_j=\sum\limits_{v\in V_j}d(v)$. Now,
  $$\langle B_\alpha {\bf x},\, {\bf x}\rangle=\langle (2\alpha-1)\,A{\bf x},\,{\bf x} \rangle+\langle (1-\alpha)\,D{\bf x},\,{\bf x}\rangle.$$
Note that
$$\langle A{\bf x},\,{\bf x}\rangle=2\sum\limits_{v_iv_j\in E(G)}x_i\,x_j=2\,(1-\chi)\,m_j+2\,(m-m_j)$$
and
$$\langle D{\bf x},\, {\bf x}\rangle=\sum\limits_{v_i\in V}d_i\,x_i^2\break=(\chi-1)^2\,m_j+(2m-m_j)=\chi\,(\chi-2)\,m_j+2m.$$
Thus we obtain
	\begin{align*}
		\langle B_\alpha {\bf x},\,{\bf x}\rangle&=(2\alpha-1)\,\langle A{\bf x},\,{\bf x} \rangle+(1-\alpha)\,\langle D{\bf x},\,{\bf x}\rangle\\
		&=2\,(2\alpha-1)\,(1-\chi)\,m_j+2\,(2\alpha-1)\,(m-m_j)+(1-\alpha)\,\chi\,(\chi-2)\,m_j+2\,(1-\alpha)\,m\\
		&=2m\,\alpha+\Big(\chi-\alpha\,(\chi+2)\Big)\,\chi m_j.
	\end{align*}
Using Rayleigh quotient, we obtain
 $$\lambda_{n}\,(B_\alpha)\,||{\bf x}||^2\leq \langle B_\alpha {\bf x},\,{\bf x} \rangle.$$
Therefore, by the above inequalities, we have	
	   $$\lambda_{n}(B_\alpha)\,\Big(\chi\,(\chi-2)\,|V_j|+n\Big)\leq 2m\alpha+\Big(\chi-\alpha\,(\chi+2)\Big)\,\chi\,m_j,$$
that is,	
		\vspace*{-1mm}
		$$\sum\limits_{j=1}^{\chi}\lambda_{n}(B_\alpha)\,\Big(\chi\,(\chi-2)\,|V_j|+n\Big)\,\leq \sum\limits_{j=1}^{\chi}\left(2m\alpha+\Big(\chi-\alpha\,(\chi+2)\Big)\,\chi \,m_j\right),$$
   that is,
 \vspace*{-1mm}
		$$\lambda_{n}(B_\alpha)\,\Big((\chi^2-2\,\chi)\,n+n\,\chi\Big)\leq 2\,\chi^2\,m-2\,\alpha \chi^2\, m-2\,m\,\chi\, \alpha.$$
	Therefore, $$\lambda_n\,(B_\alpha)\,\leq\, \frac{2\,m}{n}\,\left(\frac{\chi\,(1-\alpha)-\alpha}{\chi-1}\right).
	$$
\end{proof}

In the next couple of results, we partially characterize the graphs attaining the equality in \eqref{eq5} of Theorem \ref{Th4.4}. The proofs technique is similar to the one used in \cite{Prof.K.C.Das-1}.
\begin{theorem}\label{Th4.6}
	Let $ G $ be a bipartite graph of $n$ vertices with $m$ edges. Then, for any $\alpha \in [0,1]$,
\begin{equation} \label{eq6}
\lambda_n(B_\alpha) \leq\frac{2m}{n}\,(2-3\alpha).
\end{equation}
Equality occurs if and only if either $\alpha=\frac{2}{3}$, or $G$ is regular with $\alpha\geq \frac{1}{2}$.
\end{theorem}

\begin{proof}
	Set $\lambda_n:=\lambda_n(B_\alpha)$. For any $ {\bf x}=(x_1,x_2, \dots, x_n)^t \in \mathbb{R}^n$, by Rayleigh quotient, we obtain
	$$\lambda_n\,{\bf x}^t\,{\bf x} \leq {\bf x}^t\, B_\alpha\, {\bf x}~~\mbox{ that is, }~~\lambda_n\sum\limits_{i=1}^{n}x_i^2\leq \alpha\sum\limits_{i=1}^n\,d_i\,x_i^2+(1-2\alpha)\sum\limits_{v_iv_j\in E(G), ~i<j} (x_i-x_j)^2 .$$
Let $ V(G)=V_1\cup V_2 $ be the vertex partition of $ G $ such that no two vertices of $ V_1 $ (resp $ V_2 $) are adjacent. Take $ {\bf x}=(x_1,x_2, \dots, x_n)^t $, where the component $ x_i=1 $ if $ v_i\in V_1 $, and $x_i= -1 $ otherwise. Then
\begin{equation*}
\lambda_n(B_\alpha) \leq\frac{2m}{n}\,(2-3\alpha).
\end{equation*}
The first part of the proof is done.

\vspace*{3mm}

If the equality holds in \eqref{eq6}, then $ B_\alpha\, {\bf x} =\lambda_n{\bf x}$. Suppose $ v_i \in V_1 $ and $ v_j \in V_2 $. From the $i$-th and $j$-th equation of $B_\alpha {\bf x} =\lambda_n{\bf x}$, we obtain
  $$\lambda_n =d_i(2-3\alpha)~~\mbox{ and }~~\lambda_n=d_j(2-3\alpha),~~\mbox{ that is, }~~(\,d_i-d_j\,)(2-3\alpha)=0.$$
Therefore, for the arbitrariness of $v_i$ and $v_j$, either $\alpha=\frac{2}{3}$ or $G$ is regular. If $G$ is $r$-regular and $\alpha< \frac{1}{2}$, then
  $$\lambda_n(B_\alpha)=\lambda_n\Big((1-\alpha)D+(2\alpha-1)A\Big)=(1-\alpha)\,r+(2\alpha-1)\,\rho_1=\alpha\,r<\frac{2m}{n}\,(2-3\alpha)$$
as $\rho_1=r$ ($G$ is bipartite). Hence either $\alpha=\frac{2}{3}$, or $G$ is regular with $\alpha\geq \frac{1}{2}$.

\vspace*{3mm}

Conversely, let $\alpha=\frac{2}{3}$. Then $Q(G)=3B_{\frac{2}{3}}(G)$ and hence $\lambda_n(B_{\frac{2}{3}})=\frac{1}{3}\,q_n(G)=0=\frac{2m}{n}\,(2-3\alpha)$ as $G$ is bipartite.

\vspace*{3mm}

Let $G$ be a $r$-regular bipartite graph with $\alpha\geq \frac{1}{2}$. Then $\rho_n=-r$. Since $\alpha\geq \frac{1}{2}$, we obtain
  $$\lambda_n(B_\alpha)=\lambda_n\Big((1-\alpha)D+(2\alpha-1)A\Big)=(1-\alpha)\,r+(2\alpha-1)\,\rho_n=(2-3\,\alpha)\,r=\frac{2m}{n}\,(2-3\alpha).$$
\end{proof}

In \cite{Prof.K.C.Das-1}, the authors defined a class of graphs $\Lambda$ in the following:

Let $\Lambda$ be the class of graphs $H=(V,E)$ such that $H$ is a regular $\chi$-partite graph $(\chi\geq 3)$ with $n/\chi$ vertices in every part, where $\chi |n$, and every vertex has $\frac{d}{\chi-1}$ adjacent vertices in every other part ($d$ is the degree of each vertex in $H$).

\begin{theorem} \label{Th4.7}
	Let $ G $ be a graph with chromatic number $\chi$ such that
\begin{equation}\label{eq7}
		\lambda_n\,(B_\alpha)\,=\, \frac{2\,m}{n}\,\left(\frac{\chi\,(1-\alpha)-\alpha}{\chi-1}\right).
\end{equation}
	where $0\leq \alpha\leq 1$. Then
	\begin{itemize}
		\item[$(1)$] for $\alpha=\frac{1}{2}$, $G$ is regular,
        \item[$(2)$] for $\chi=2$, $G$ is bipartite with $\alpha=\frac{2}{3}$, or $G$ is regular bipartite with $\alpha\geq \frac{1}{2}$.
		\item [$(3)$] for $\alpha\in [0,1]\,\Big(\alpha \neq \, \Big\{\frac{1}{2},\, \frac{\chi}{\chi+1}\Big\},\,\chi\geq 3\Big)$, $G \in \Lambda$.
	\end{itemize}
\end{theorem}

\begin{proof} (1) Suppose that $\alpha=\frac{1}{2}$. Then by Remark 1.1 and (\ref{eq7}), we obtain $\frac{\delta}{2}=\lambda_n(B_{\frac{1}{2}})=\frac{m}{n}$, that is, $2m=n\,\delta$, that is, $n\delta\leq \sum\limits^n_{i=1}\,d_i=2m=n\delta$, that is, $\sum\limits^n_{i=1}\,d_i=n\delta$, that is, $G$ is regular.
	
	\vspace*{4mm}

\noindent
(2) Suppose $ \chi=2$. Then $\lambda_n(B_\alpha)=\frac{2m}{n}\,(2-3\alpha)$. By Theorem \ref{Th4.6}, $G$ is bipartite with $\alpha=\frac{2}{3}$, or $G$ is regular bipartite with $\alpha\geq \frac{1}{2}$.
	
	\vspace*{2mm}

\noindent
(3) We assume that $\alpha\in [0,1]$ and $\alpha \neq \Big\{\frac{1}{2},\, \frac{\chi}{\chi+1}\Big\}$ with $\chi\geq 3$.

Set $ \lambda_n:=\lambda_n(B_\alpha) $. Let us partition the vertex set $ V(G) $ into $ \chi $ number of color classes $ V_1, V_2, \dots, V_\chi $. For $ j\in \{ 1,2 \dots, \chi\} $, define $ {\bf x^{(j)}}:=(x^j_1,\,x^j_2, \ldots,\,x^j_n) $ as follows:
	
	\begin{center}
		$x^j_i = \left\{ \begin{array}{cc}
			\chi-1 & \mbox{for $v_i \in V_j$,} \\[3mm]
			-1 & \mbox{otherwise.}
		\end{array}\right.$
	\end{center}	
Then by Theorem \ref{Th4.4}, $\lambda_n\,(B_\alpha)\,\leq\, \frac{2\,m}{n}\,\left(\frac{\chi\,(1-\alpha)-\alpha}{\chi-1}\right) $. Since equality occurs in the above inequality, so by Rayleigh quotient and Theorem \ref{Th4.4}, $ {\bf x^{(1)}}, \dots, {\bf x^{(\chi)}} $	 are all eigenvectors of $ B_\alpha $ corresponding to the eigenvalue $ \lambda_n$.
	
	\vspace*{3mm}

\noindent
${\bf Claim\,1:}$ $ G $ is regular.\\
	Let $ 1\leq k \ne \ell\leq \chi $. Suppose $ v_s \in V_k $ and $ v_t \in V_{\ell} $. Comparing the $ s $-th components of the matrix equation $ B_\alpha {\bf x^{(k)}}=\lambda_n {\bf x^{(k)}}$, we obtain $ \lambda_n(\chi-1)=d_s\,(\chi-\alpha\chi-\alpha) $. Similarly, comparing the $ t $-th components of the matrix equation  $ B_\alpha {\bf x}^{(\ell)}=\lambda_n {\bf x}^{(\ell)}$, we have $ \lambda_n(\chi-1)=d_t\,(\chi-\alpha\chi-\alpha) $. Then $ (d_s-d_t)(\chi-\alpha\chi-\alpha)=0.$ Since $ \alpha \ne \frac{\chi}{\chi+1} $ and $ v_s $, $ v_t $ are arbitrary, so $ G $ is regular.
	
	\vspace*{3mm}
	
\noindent
${\bf Claim\,2:}$ $ |V_1|=\cdots=|V_\chi|=\frac{n}{\chi}$, where $ \chi|n $. \\
	By Claim 1, $ G $ is regular, so $ \boldsymbol{1}:=(1,1, \dots, 1)^t $ is an eigenvector of $ B_\alpha $. Also, $B_\alpha$ is symmetric, so $\boldsymbol{1} \perp x^{(k)}$ for $k=1,2, \dots, \chi$. Therefore, $|V_k|(\chi-1)+(-1)(n-|V_k|)=0$. That is, $|V_k|=\frac{n}{\chi}$, for $k=1,2, \dots, \chi$.

\vspace*{3mm}
	
\noindent
${\bf Claim\,3:}$ Every vertex is adjacent to $ \frac{d}{\chi-1} $ vertices in every other part, where $ d $ is the regularity of $ G $.\\
	Suppose $v_s\in V_1$ and it is adjacent with $r_2,\,r_3,\ldots,\,r_\chi$ number of vertices in the partitions $V_2,\,V_3,\ldots,\,V_\chi $, respectively. For $v_s\in V_1$, from $B_\alpha {\bf x^{(k)}}=\lambda_n\,{\bf x^{(k)}}$, we obtain
  $$(-1)\lambda_n=(1-\alpha)(-1)d_s+(2\alpha-1)\,\left( \sum\limits^{\chi}_{i=2,i\neq k}\,(-1)r_i+(\chi-1)r_k\right),$$
where $2\leq k\leq \chi$. From this we have the following $\chi-1$ equations:
\begin{align*}
		(-1)\lambda_n&=(1-\alpha)(-1)d_s+(2\alpha-1)\Big( (\chi-1)r_2+(-1)r_3+\cdots+(-1)r_\chi\Big)\\
		(-1)\lambda_n&=(1-\alpha)(-1)d_s+(2\alpha-1)\Big( (-1)r_2+(\chi-1)r_3+\cdots+(-1)r_\chi\Big)\\
		\cdots & \cdots \cdots\\
		\cdots & \cdots \cdots\\
		(-1)\lambda_n&=(1-\alpha)(-1)d_s+(2\alpha-1)\Big( (-1)r_2+(-1)r_3+\cdots+(\chi-1)r_\chi\Big). 	
\end{align*}
	Since $ \alpha \ne \frac{1}{2} $, so from the above, we have $ r_2=r_3=\cdots=r_\chi= \frac{d}{\chi-1}$ as $G$ is regular by ${\bf Claim\,1}$. Also $ v_s $ is arbitrary, therefore the ${\bf Claim\,3}$ is done.

\vspace*{3mm}

Hence $G\in \Lambda$.	
\end{proof}

In the next result, we partially obtain the converse of the Theorem \ref{Th4.7}. One can verify that if $ \alpha=\frac{1}{2} $ and $ G $ is regular, then the equality \eqref{eq7} holds. Moreover, the equality \eqref{eq7} holds for any bipartite graph with $\alpha=\frac{2}{3}$, or any regular bipartite graph with $\alpha\geq \frac{1}{2}$. Therefore, we consider the remaining converse part of the Theorem \ref{Th4.7}. Since the proof technique of the following result is similar to \cite[Theorem 5.1]{Prof.K.C.Das-1}, so we skip it for the readers.
\begin{theorem}
	If $ G \in \Lambda $ and $\alpha\,(\ne \frac{1}{2}) \in [0,1]$, then $\displaystyle{\frac{2m}{n}\left(\frac{\chi(1-\alpha)-\alpha}{\chi-1}\right)}$ is an eigenvalue of $ B_\alpha(G) $ with multiplicity $ \chi-1 $.
\end{theorem}


The next result is known (see to \cite{Signless-Chromatic}); however, it can be deduced from Theorem \ref{Th4.4} by taking $ \alpha=\frac{2}{3}$.
\begin{corollary}{\rm \cite{Signless-Chromatic}}
	Let $ G $ be graph of order $ n $ with $ m $ edges and chromatic number $ \chi $. Then $$ q_n(G)\leq \frac{2m}{n}\left(\frac{\chi-2}{\chi-1}\right).$$
\end{corollary}

\begin{corollary}
	Let $ G $ be graph of order $ n $ with $ m $ edges and chromatic number $ \chi $ such that \begin{equation*}
	 q_n(G)=
	 \frac{2m}{n}\left(\frac{\chi-2}{\chi-1}\right).
	\end{equation*}
	Then $ G $ is either bipartite or $ G \in \Lambda $.
\end{corollary}
\begin{proof}
	Proof follows from Theorem \ref{Th4.7}.
\end{proof}

As a consequence of Theorem \ref{Th4.4}, a lower bound of the chromatic number of $ G $ is deduced.
\begin{corollary}
	If $ G $ is a graph with chromatic number $ \chi $ and $ \beta_o:=\beta_o(G) $, then $$ \chi\geq \frac{\beta_o}{1-\beta_o}.$$
\end{corollary}


\section{On the determinant}\label{sec3}
In this short section, we present the determinant and the Sachs-type formula for the coefficients of the characteristic polynomial of $ B_\alpha(G) $. Then, we obtain some known results as a consequence.

A spanning elementary subgraph $ H $ of a graph $ G $ is a spanning subgraph of $ G $ such that each component of $ H $ is either a cycle or an edge. For a spanning elementary subgraph $ H $, $ p(H) $ and $ c(H) $ denote the number of components and the number of cycles in $ H $, respectively. Now, we present the well known Harary's formula \cite{Harary} for the determinant of the adjacency matrix of a graph.

\begin{prop}\label{Prop5.1}
	Let $G$ be a graph with $V(G)=\{v_1,v_2,\ldots,v_n\}$. Also let $ A(G) $ be the adjacency matrix of $ G $. Then,
	\begin{align*}
		\det\,(A(G))=\sum\limits_{H}\,(-1)^{\,n-p(H)}\,2^{\,c(H)},
	\end{align*}
	where summation is over all spanning elementary subgraphs $ H $ of $ G $.
\end{prop}


Motivated by the notion of spanning elementary subgraphs and for the purpose of the main result in this section, we define the following.

\begin{definition}{Modified Elementary Subgraph:} A subgraph $H$ of a graph $G$ is called a modified elementary subgraph if each component of $H$ is either a vertex, an edge, or a cycle.
\end{definition}

Let $H$ be a modified elementary subgraph. Denote by $c(H), c_1(H),$ and  $c_2(H)$ the number of components in a subgraph $ H $ which are cycles, edges, and vertices, respectively. Let $ p(H):=c(H)+c_1(H)+c_2(H) $ be the number of components in $ H $. Also, let $ \mathcal{C}_2(H) $ be the collection of isolated vertices in $ H $. In the following result, we present a Harary-type formula \cite{Harary} for the determinant of $ B_\alpha $-matrices. For a graph $ G $, since $ \det(L(G))=\det(B_0(G))=0 $, so we derive a formula of $ \det(B_\alpha(G)) $ for $ \alpha \in (0,1] $.

\begin{theorem}\label{th3.1}
	Let $G$ be a graph with $V(G)=\{v_1,v_2,\ldots,v_n\}$. Then, for any $ \alpha \in (0,1] $,
	\begin{equation}\label{det}
		\det\,(B_{\alpha}(G))=\sum\limits_{H}\,(-1)^{\,n-p(H)}\,2^{\,c(H)}\,(1-\alpha)^{\,c_2(H)}\,(2\alpha-1)^{\,n-c_2(H)}\left(\prod_{v_i\in \mathcal{C}_2(H)}d_G(v_i)\right),
	\end{equation}
	where summation is over all spanning modified elementary subgraphs $ H $ of $ G $.
\end{theorem}

\begin{proof}
Consider $B_{\alpha}(G)=(2\alpha-1)A(G)+(1-\alpha)D(G)=(b_{ij})_{n\times n}$. We have
\begin{align}\label{eq2.1}
\det\,(B_{\alpha}(G))=\sum\limits_{\pi}\,\text{sgn}(\pi)\,b_{1\pi(1)}\,b_{2\pi(2)}\,\cdots\, b_{n\pi(n)},
\end{align}
where summation is over all permutations of $1,2,\ldots,n$.
Since every permutation $ \pi $ has a cycle decomposition, so a cycle of length $ 1 $, $ 2 $ and more corresponds to a vertex, an edge, and a cycle, respectively in the graph $ G $. Thus, each term $b_{1\pi(1)}b_{2\pi(2)}\cdots b_{n\pi(n)}$ corresponds to a spanning modified elementary subgraph of $G$.

\vspace*{3mm}
Also, each spanning modified elementary subgraph $H$ corresponds to $2^{c(H)}$ terms in the summation \eqref{eq2.1} as each cycle is associated to a cyclic permutation in two ways. If $\pi$ is a permutation corresponding to a spanning modified elementary subgraph $H$, then
\begin{align*}
\text{sgn}\,(\pi)&=(-1)^{\,n-\text{number of cycles in the cyclic decomposition of }\pi}\\
&=(-1)^{\,n-c(H)-c_1(H)-c_2(H)}\,=\,(-1)^{\,n-p(H)}
\end{align*}
and
\begin{align*}
b_{1\pi(1)}\,b_{2\pi(2)}\,\cdots\, b_{n\pi(n)}&=(1-\alpha)^{c_2(H)}\,\left(\prod_{v_i\in \mathcal{C}_2(H)}d_G(v_i)\right)\,(2\alpha-1)^{2c_1(H)}\,(2\alpha-1)^{n-2c_1(H)-c_2(H)}\\[2mm]
&=(1-\alpha)^{c_2(H)}\left(\prod_{v_i\in \mathcal{C}_2(H)}d_G(v_i)\right)(2\alpha-1)^{n-c_2(H)}.
\end{align*}
Therefore,
\begin{align*}
\det\,(B_{\alpha}(G))=\sum\limits_H(-1)^{\,n-p(H)}\,2^{\,c(H)}\,(1-\alpha)^{c_2(H)}\,(2\alpha-1)^{\,n-c_2(H)}\,\left(\prod_{v_i\in \mathcal{C}_2(H)}d_G(v_i)\right),
\end{align*}
where summation is over all spanning modified elementary subgraphs $H$ of $G$ .
\end{proof}

In the next corollary, we obtain a Sachs-type formula for the coefficients of the characteristic polynomial of $B_{\alpha}(G)$.

\begin{corollary}\label{Cor5.1}
Let $\phi(B_{\alpha})=\lambda^n+a_1\lambda^{n-1}+a_2\lambda^{n-2}+\cdots+a_{n-1}\lambda+a_n$ be the characteristic polynomial of $ B_\alpha(G) $. Then
	\begin{align*}
		a_k&=\sum\limits_{H}\,(-1)^{\,p(H)}\,2^{\,c(H)}\,(1-\alpha)^{c_2(H)}\, (2\alpha-1)^{\,n-c_2(H)}\left(\prod_{v_i\in \mathcal{C}_2(H)}d_G(v_i)\right),
	\end{align*}
	where summation is over all modified elementary subgraphs $H$ of $G$ with $k$ vertices.	
\end{corollary}

\begin{proof}
The proof follows from Theorem \ref{th3.1} and recalling that $c_k$ is $(-1)^k$ times the sum of $k\times k$ principal minors of $B_{\alpha}(G)$.
\end{proof}

%

One can observe that Proposition \ref{Prop5.1} can also be deduced as a consequence of the Theorem \ref{th3.1}.
\begin{corollary}
	Let $G$ be a graph with $V(G)=\{v_1,v_2,\ldots,v_n\}$. Also let $ Q(G) $ be the signless Laplacian matrix of $ G $. Then,
	\begin{align*}
		\det\,(Q(G))=\sum\limits_{H}\,(-1)^{\,n-p(H)}\,2^{\,c(H)}\left(\prod_{v_i\in \mathcal{C}_2(H)}d_G(v_i)\right),
	\end{align*}
	where summation is over all spanning modified elementary subgraphs $ H $ of $ G $.
\end{corollary}

\begin{proof}
Setting $\alpha=\frac{2}{3}$ in the formula \eqref{det} of Theorem \ref{th3.1}, we obtain the result.
\end{proof}

\vspace*{3mm}

\section*{Declaration of competing interest}
There are neither conflicts of interest nor competing interests. The authors have no relevant financial interests.

\section*{Acknowledgements}
A. Samanta thanks the National Board for Higher Mathematics (NBHM), Department of Atomic Energy, India, for financial support in the form of an NBHM Post-doctoral Fellowship (Sanction Order No. 0204/21/2023/R\&D-II/10038). The first author also acknowledges excellent working conditions in the Theoretical Statistics and Mathematics Unit, Indian Statistical Institute Kolkata. K. C. Das is supported by National Research Foundation funded by the Korean government (Grant No. 2021R1F1A1050646).


\mbox{}

\end{document}